\documentclass[a4paper,11pt,twoside]{amsart}

\usepackage[latin1]{inputenc}
\usepackage{amsmath, amsfonts, amssymb,amsthm}
\usepackage[mathscr]{eucal} 
\usepackage[pdftex]{graphicx}
\usepackage{color}

\usepackage[T1]{fontenc}
\usepackage[sc]{mathpazo}
\usepackage[all]{xy}
\usepackage{hyperref}
\usepackage[flenq]{nccmath}

\usepackage {anysize}
\marginsize{2cm}{2cm}{2cm}{3cm}

\usepackage{enumerate}

\definecolor{alert}{rgb}{0.4,0.9,0.4}

\newcommand{\R}{\mathbb{R}}

\newcommand{\s}{\mathbb{S}}
\newcommand{\h}{\mathbb{H}}

\newcommand{\fl}{\longrightarrow}
\newcommand{\oz}{\overline{z}}

\newcommand{\m}{M^2(\epsilon)}

\newcommand{\real}{\mathbb{R}}

\newcommand{\me}{M^2(\epsilon)\times\R}

\newtheorem{theorem}{Theorem}[section]
\newtheorem{proposition}[theorem]{Proposition}

\newtheorem{lemma}[theorem]{Lemma}

\theoremstyle{definition}
  \newtheorem{definition}[theorem]{Definition}

\theoremstyle{remark}
\newtheorem{remark}[theorem]{Remark}

\numberwithin{equation}{section}

\title[Weingarten type surfaces]{Weingarten type surfaces in $\h^2\times\R$ and $\s^2\times\R$}

\author{Abigail Folha and Carlos Pe\~{n}afiel}
\address{}
\email{}

\thanks{This work was partially supported by CNPq, Conselho Nacional de Desenvolvimento Cient{\'i}fico e Tecnol{\'o}gico - Brazil.}



\begin{document}

\begin{abstract}
In this article, we study complete surfaces $\Sigma$, isometrically immersed  in the product space $\h^2\times\R$ or $\s^2\times\R$ having positive extrinsic curvature $K_e$. Let $K_i$ denote the intrinsic curvature  of $\Sigma$. Assume  that the equation $aK_i+bK_e=c$ holds for some real constants $a\neq0$, $b>0$ and $c$. The main result of this article state that when such a surface is a topological sphere it is rotational. 
 \end{abstract}

\maketitle

\section{Introduction}
The Hopf's holomorphic quadratic differential form, defined on surface having constant mean curvature in $\R^3$, enables Hopf to give a proof that topological spheres in $\R^3$ having constant mean curvature are rotational. A few years ago, Abresch and Rosenberg (see  \cite{AR}, \cite{UR}) discovered a holomorphic quadratic differential form on constant mean curvature surfaces in the homogeneous 3-manifolds. With the aid of this quadratic form, they extended the Hopf's result to constant mean curvature topological spheres immersed in such homogeneous spaces.

In the product spaces $\h^2\times\R$ and $\s^2\times\R$, Aledo, Espinar and G\'{a}lvez \cite{AEG1}, associated a holomorphic quadratic differential form to constant intrinsic curvature (Gaussian curvature) surfaces immersed in such product spaces, which enabled them to extend the classical Liebmann Theorem, that in the euclidean space $\R^3$ ensure that, the round spheres are the unique complete surfaces of positive constant intrinsic curvature. For complete surfaces immersed in $\h^2\times\R$ or $\s^2\times\R$ having positive extrinsic curvature, G\'{a}lvez, Espinar and Rosenberg \cite{EGH}, proved that such surfaces are embedded and homeomorphic to either the euclidean sphere $\s^2$ or to the euclidean plane $\R^2$. Moreover, they construct a quadratic differential form on positive constant extrinsic surfaces which vanishes identically or its  zeros are isolated with negative index. As a consequence, they proved that the complete immersions having positive constant extrinsic curvature in the product spaces $\h^2\times\R$ and $\s^2\times\R$ are  rotational sphere.

In this article, we consider complete surfaces $\Sigma$, isometrically immersed in the product spaces $\h^2\times\R$ or $\s^2\times\R$ having positive extrinsic curvature (non-constant) such that, 
$$aK_i+bK_e=c,$$
where $K_i$, $K_e$ are the intrinsic and the extrinsic curvatures, respectively, and $a\neq0$, $b>0$ and $c$ are real constants. Our goal is to prove that if $\Sigma$ is a topological sphere, then $\Sigma$ is rotational. In order to obtain this result, we first construct a quadratic differential form $Qdz^2$ which vanishes identically or its zeros are isolated with negative index, this quadratic form exists if $a+b\neq0$, $2a+b\neq0$, see subsection \ref{quadratic}. Moreover, we obtain vertical and horizontal height estimates which enable us to realize when $\Sigma$ is a topological sphere, see Section \ref{HVE} and Section \ref{PES}. In Section \ref{SMT}, we prove the main theorem.

The article is organized as follows: In section \ref{preliminares}, we give the definition of Weingarten type surfaces.  Section \ref{rotational} is devoted to the study of rotational examples. In section \ref{HVE}, we construct a quadratic differential form on a Weingarten type surface which vanishes identically or its zeros are isolated with negative index. Also, we stablish horizontal and vertical height estimates. In Section \ref{PES} we study the non-existence of properly embedded surfaces having finite topology a one top end. In section \ref{SMT}, we prove the main theorem of this article.

Finally, the authors would like to thank Jos\'{e} Antonio G\'{a}lvez for his guidance and encouragement during the preparation of this paper.

\section{Weingarten type Surfaces having positive extrinsic curvature}\label{preliminares}

For $\epsilon\in\{-1,1\}$, we denote by $\m$ the complete, connected, simply-connected, 2-dimensional space form having sectional curvature $\epsilon$. That is, for $\epsilon=1$, $\m$ label the canonical euclidean unit sphere $\s^2$ and for $\epsilon=-1$, $\m$ denotes the complete, connected, simply-connected hyperbolic plane $\h^2$ having sectional curvature $-1$. Also, we denote by $\me$ the product space (where $\R$ is the real line), endowed with the product metric.

Recall that, the surface $\Sigma$ is called a Weingarten surface if its two principal curvatures $k_1$ and $k_2$ are not independent one of another or, equivalently, if there exist a relation of the form $W(k_1,k_2)=0$ for a smooth real function $W:\mathcal{D}\subset\R^2 \fl\R$ defined on a set $\mathcal{D}$.

In this article, we study complete, connected, surfaces $\Sigma$ isometrically immersed in the product space $\me$ whose intrinsic and extrinsic curvature are linearly related, more precisely,

\begin{definition}\label{d1}
Let $\varphi : \Sigma\fl \me$ be an isometric immersion from a connected surface having intrinsic curvature $K_i$ and extrinsic curvature $K_e$. We say that  $\Sigma$ is a Weingarten type surface or simply a $W$-surface if there exist three real numbers, $a\neq0$, $b>0$ and $c$ such that,
\begin{equation}\label{DE1}
aK_i+bK_e-c=0.
\end{equation}
\end{definition}
\begin{remark}
The assumption $b>0$ is not a restriction since we can multiply by -1 the equation \ref{DE1} if necessary. 
\end{remark}

For simplicity, we treat properties of an immersion $\varphi$ as those of $\Sigma$ and denote merely by $\Sigma$ the image $\varphi(\Sigma)$. For example, we call $\Sigma$ a W-surface in $\me$
instead of saying that the immersion $\varphi$ is a W-surface in $\me$.

Let $\varphi : \Sigma\fl \me$ be an isometric immersion from an orientable surface $\Sigma$ into the product space $M^2(\epsilon)\times\R$. We chose a global unit normal vector field $N$ and as usual, we denote by $\nu=\langle N,\frac{\partial}{\partial t}\rangle$ the angle function of $\Sigma$, here $\frac{\partial}{\partial t}$ denotes the tangent vector field to the real line $\R$. From \cite{Benoit}, we have that the Gauss equation for such an immersed surface into the product space $\me$ is given by
\begin{equation}\label{G1}
K_i=K_e+\epsilon\nu^2.
\end{equation}

As a consequence of the Gauss equation, we have

\begin{lemma}\label{c12}
Let $\varphi : \Sigma\fl \me$ be an isometric immersion. Assume that $\Sigma$ is a complete W-surface having positive extrinsic curvature $K_e$. Then: 
\begin{enumerate}
\item Suppose $a+b>0$, then
\begin{itemize}
\item[i)] If $c\le0$, $\Sigma$ is not compact.
\item[ii)] If $\epsilon=-1$ and $c>b$, $\Sigma$ is closed.
\item[iii)] If $\epsilon=1$ and $c>0$, $\Sigma$ is closed.
\end{itemize}
\item Suppose $a+b<0$, then
\begin{itemize}
\item[i)] If $c\ge0$, $\Sigma$ is not compact.
\item[ii)] If $\epsilon=-1$ and $c<0$, $\Sigma$ is closed.
\item[iii)] If $\epsilon=1$ and $c<-b$, $\Sigma$ is closed.
\item[iv)] If $\epsilon=1$ and $-b\le c<0$, $\Sigma$ cannot be closed.
\end{itemize}
\item For $a+b=0$, the angle function is constant.
\end{enumerate}
\end{lemma}
\begin{proof}
As the extrinsic curvature of the W-surface is positive,  $\Sigma$ is orientable and we choose the unit global normal vector field such that the second fundamental form is definite positive. 

For $\h^2\times\R$, if $a+b<0$, it is clear that for $c<0$, the intrinsic curvature satisfies $K_i\ge\frac{c}{a+b}>0$. Then, from the Bonnet-Myers theorem, $\Sigma$ must be compact. 
On the other hand, for  $a+b<0$ and $c\ge0$, if $\Sigma$ were compact, there would exist a point $p\in\Sigma$ such that $\nu(p)=0$, it would imply that the extrinsic curvature satisfies $K_e(p)\le0$, which contradicts our assumption. The proof of the other cases are similar.

From equation \eqref{DE1} and \eqref{G1}, we conclude that $a+b=0$ implies that the angle function is constant. 
\end{proof}
\begin{remark}
Since surfaces having constant angle were treated in \cite{D1} and \cite{D2}, from now on, we omit this case.
\end{remark}

\section{Complete rotational surfaces of Weingarten type in $M^2(\epsilon)\times\R$}\label{rotational}

In this section, we deal with complete $W$-surfaces having positive extrinsic curvature, which are invariant by one-paremeter group of rotations of the ambient space $M^2(\epsilon)\times\R$.

For $\epsilon\in\{-1,1\}$, let us consider the 4-dimensional space $\mathbb{R}^3_\epsilon\times\R$, endowed with the metric $ds^2=\epsilon dx_1^2+dx_2^2+dx_3^2+dx_4^2$. And, let us identify the product space $M^2(\epsilon)\times\R$ as being the sub-manifold of $\mathbb{R}^3_\epsilon\times\R$, given by
$$M^2(\epsilon)\times\R=\{(x_1,x_2,x_3,x_4)\in \mathbb{R}^3_\epsilon\times\R:\epsilon x_1^2+x_2^2+x_3^2= \epsilon, \hspace{.1cm} \textnormal{and if} \hspace*{.2cm} \epsilon=-1, \hspace{.2cm} x_1>0\}.$$

The rotation in $\me$ is a subgroup of the isometry group of $\me$ which preserves the orientation and fixes an axis $\{p\}\times\R$, with $p\in M^2(\epsilon)\times\{0\}$. This subgroup can be identify with the special orthogonal group $SO(2)$. Up to isometries, we can assume that the axis is given by $\{(1, 0, 0)\}\times\R$.

We consider the plane $\Pi=\{(x_1,x_2,0,x_4)\in M^2(\epsilon) \times\R,x_2\ge0\}$ and the curve 
\begin{equation*}
\alpha_\epsilon(u)=\left\{
\begin{array}{lcr}
(\cosh k(u), \sinh k(u), 0, h(u)) & \subset & \Pi, \hspace{.3cm} \textnormal{if} \hspace{.3cm} \epsilon=-1, \vspace{.3cm}\\
(\cos k(u), \sin k(u), 0, h(u)) & \subset & \Pi, \hspace{.5cm} \textnormal{if} \hspace{.6cm} \epsilon=1.
\end{array}
\right.
\end{equation*}
Where $k(u)\ge0$ and $u$ is the arclength of $\alpha$, that is, $(k^\prime(u))^2+(h^\prime(u))^2=1$. Here $k^\prime(u)$ denotes the derivative with respect to the variable $u$.

In order to obtain a rotational surface, we apply one-parameter group of rotational isometries to the curve $\alpha_\epsilon$. Denoting by $\mathcal{S}$ such a generated surface, we can parametrized  $\mathcal{S}$ by
\begin{equation*}
\varphi_\epsilon(u,v)=\left\{
\begin{array}{lcr}
(\cosh k(u), \sinh k(u) \cos v, \sinh k(u) \sin v, h(u)), &  & \textnormal{if} \hspace{.3cm} \epsilon=-1, \vspace{.3cm}\\
(\cos k(u), \sin k(u) \cos v, \sin k(u) \sin v, h(u)), &  & \textnormal{if} \hspace{.5cm} \epsilon=1.
\end{array}
\right.  
\end{equation*}

In order to simplify the expressions, we define the functions
\begin{equation*}
\cos_\epsilon k=\left\{
\begin{array}{lcr}
\cosh k,  &  & \textnormal{if} \hspace{.3cm} \epsilon=-1 \vspace{.3cm}\\
\cos k, &  & \textnormal{if} \hspace{.5cm} \epsilon=1
\end{array}
\right.   \hspace{.4cm} \textnormal{and} \hspace{.4cm}
\cot_\epsilon k=\left\{
\begin{array}{lcr}
\coth k,  &  & \textnormal{if} \hspace{.3cm} \epsilon=-1 \vspace{.3cm}\\
\cot k, &  & \textnormal{if} \hspace{.5cm} \epsilon=1.
\end{array}
\right.
\end{equation*}

\subsection{The first integral.}\label{HSphere} 
The aim of this section is to classify complete rotational W-surfaces having positive extrinsic curvature. A straightforward computation gives us that the intrinsic and extrinsic curvature function of an isometrically immersed surface which is invariant by rotational isometries in the space $\me$ are given by
\begin{eqnarray*}
K_i & = & \epsilon (k^\prime(u))^2 - k^{\prime\prime}(u) \cot_\epsilon k(u), \\
K_e & = & -k^{\prime\prime}(u) \cot_\epsilon k(u).
\end{eqnarray*}
The Weingarten equation is written as
\begin{equation}\label{e21}
(a+b)k^{\prime\prime}(u) \cot_\epsilon k(u)-\epsilon a (k^\prime(u))^2 =-c,
\end{equation}
for real numbers $a\neq0$, $b>0$ and $c$ satisfying $a+b\neq0$. It is direct to check that the first integral of the ordinary differential equation (ODE in short) \eqref{e21} is
\begin{equation}\label{e22}
(k^\prime(u))^2=\epsilon\dfrac{c}{a}+C_1(\cos_\epsilon k(u))^{-\frac{2a}{a+b}}
\end{equation}
for some constant $C_1$. Moreover, we can assume that  $\alpha$ cuts the axis orthogonally at $t = 0$. Then $k(0) = 0$ and $k^\prime(0) = 1$, in this case, the first integral is given by
\begin{equation}\label{e23}
(k^\prime(u))^2=\epsilon\dfrac{c}{a}+\dfrac{a-\epsilon c}{a}(\cos_\epsilon k(u))^{-\frac{2a}{a+b}}.
\end{equation}
Notice that, the problem of find all complete rotational W-surfaces which cut the axis orthogonally, consist in determine all the admissible expressions of the profile curve $\alpha_\epsilon$, we mean, we wish to find all the possible compact (and non-compact) integral curves of the ODE system
\begin{equation}\label{e24}
\left\{
  \begin{array}{rcl}
    (k^\prime(u))^2 -\epsilon\dfrac{c}{a} &=& \dfrac{a-\epsilon c}{a}(\cos_\epsilon k(u))^{-\frac{2a}{a+b}}, \vspace{.3cm} \\ 
    (k^\prime(u))^2+(h^\prime(u))^2 &=& 1.
  \end{array}
\right.
\end{equation}
We have the next proposition.

\begin{proposition}\label{p1}
Let $\mathcal{S}$ be a complete W-surface isometrically immersed into the product space $\me$ having positive extrinsic curvature $K_e$ which is invariant by rotational isometries and whose generating curve $\alpha_\epsilon$ cuts the rotation axis orthogonally. Assume that $a+b\neq0$, then: 
\begin{enumerate}
\item For $a+b>0$, there are two cases,
\begin{itemize}
\item If $c>0$, $\mathcal{S}$ is a rotational topological sphere.
\item If $c\le0$, $\mathcal{S}$ homeomorphic to $\R^2$.
\end{itemize}
\item For $a+b<0$, there are two cases,
\begin{itemize}
\item If $c\ge0$, $\mathcal{S}$ is homeomorphic to $\R^2$.
\item In $\h^2\times\R$, if $c<0$, $\mathcal{S}$ is a rotational topological sphere.
\item In $\s^2\times\R$, if $c<-b$, $\mathcal{S}$ is a rotational topological sphere. 
\item In $\s^2\times\R$, if $-b\le c<0$, there is no rotational surface $\mathcal{S}$ whose generating curve cuts orthogonally the rotation axis.
\end{itemize}
\end{enumerate}
\end{proposition}
\begin{proof} It is known that complete surfaces, isometrically immersed in the product spaces $\me$, having positive extrinsic curvature are  homeomorphic  either to a sphere or to the euclidean plane $\R^2$, see \cite[Theorem 3.1]{EGH}, and \cite[Theorem 2.4]{RE}. By Lemma \ref{c12}, in order to prove the proposition, we just need to consider two cases. The first one is $\epsilon=-1$, $a+b>0$ and $c>0$, and the second is $\epsilon=1$, $a+b<0$ and $-b\le c<0$. 

From \eqref{e24}, for $\epsilon=-1$, $a+b>0$ and $c>0$, we have
\begin{equation}\label{e25}
\left(\dfrac{dh}{dk}\right)^2=\left(\dfrac{a+c}{c}\right) \dfrac{(\cosh k)^{\frac{2a}{a+b}}-1}{\left(\sqrt{\frac{a+c}{c}}+(\cosh k)^{\frac{a}{a+b}}\right)\left(\sqrt{\frac{a+c}{c}}-(\cosh k)^{\frac{a}{a+b}}\right)}.
\end{equation}
That is, we consider the function $h=h(k)$ as a function of the variable $k$; notice that $k$ is the hyperbolic distance to the origin in the slice $\h^2\times\{0\}$. This function is defined on the interval $[0,k_0]$, where $k_0$ satisfies
$$(\cosh k_0)^{\frac{a}{a+b}}=\sqrt{\dfrac{a+c}{c}}.$$
The graph of the function $h$ has a vertical tangent line at $k_0$. In order to obtain a rotational topological sphere, we need to show that the height function $h(k)$ is bounded and it is of class $C^2$ at $k=k_0$.

Up to isometries of the ambient space, we can assume that $\dfrac{dh}{dk}\ge0$. We separate the proof in two cases, depending on the sign of $a$ (recall we are assuming that $a\neq0$).
\begin{itemize}
\item[i)] If $a>0$, then $\dfrac{2a}{a+b}>0$ and $(\cosh k)^{\frac{2a}{a+b}}-1>0$. For $k>0$, we set 
$$A_1(k)=\sqrt{\left(\dfrac{a+c}{c}\right) \dfrac{(\cosh k)^{\frac{2a}{a+b}}-1}{\left(\sqrt{\frac{a+c}{c}}+(\cosh k)^{\frac{a}{a+b}}\right)}},$$
thus, equation \eqref{e25} implies
\begin{equation}\label{e38}
\dfrac{dh}{dk}= A_1(k)\dfrac{1}{\sqrt{\sqrt{\frac{a+c}{c}}-(\cosh k)^{\frac{a}{a+b}}}}\dfrac{2(\frac{-a}{a+b})(\cosh k)^{\frac{-b}{a+b}}\sinh k}{2(\frac{-a}{a+b})(\cosh k)^{\frac{-b}{a+b}}\sinh k}
\end{equation}
moreover, if we consider the function
$$A_2(k)=A_1(k)\dfrac{2}{(\frac{a}{a+b})(\cosh k)^{\frac{-b}{a+b}}\sinh k},$$
we can write equation \eqref{e38} as
\begin{equation}\label{e26}
\dfrac{dh}{dk}= A_2(k)\dfrac{(-1)(\frac{-a}{a+b})(\cosh k)^{\frac{-b}{a+b}}\sinh k}{2\sqrt{\sqrt{\frac{a+c}{c}}-(\cosh k)^{\frac{a}{a+b}}}}=A_2(k)\dfrac{d}{dk}\sqrt{\sqrt{\frac{a+c}{c}}-(\cosh k)^{\frac{a}{a+b}}}.
\end{equation}
Notice that $A_1(k)$ and $A_2(k)$ are bounded functions on the iterval $[0,k_0]$, then, for each $0<\delta<k_0$, there exist a positive number $M>0$, such that, for all $k\in[k_0-\delta,k_0]$, we have
\begin{equation}\label{e27}
\dfrac{dh}{dk}\le -M\dfrac{d}{dk}\sqrt{\sqrt{\frac{a+c}{c}}-(\cosh k)^{\frac{a}{a+b}}}.
\end{equation}
Integrating \eqref{e27}, there exists a constant $C_1$ large enough, such that
\begin{equation}\label{e28}
0<h(k)\le M\left(C_1-\sqrt{\sqrt{\frac{a+c}{c}}-(\cosh k)^{\frac{a}{a+b}}}\right).
\end{equation}
The function $h=h(k)$ is bounded in $[k_0-\delta,k_0]$. From equation \eqref{e25} its graph has a vertical tangent line at $k=k_0$, a straightforward computation gives that the  function $h$ is of class $C^2$ at $k=k_0$, that is, its graph  has bounded curvature at $k=k_0$. So after a reflection about the slice $t=h(k_0)$, we obtain a complete rotational topological sphere.
\item[ii)] The proof for the case $a<0$ is analogous, taking into account that in this case $\dfrac{2a}{a+b}<0$ and $(\cosh k)^{\frac{2a}{a+b}}-1<0$.
\end{itemize}

For the case $\s^2\times\R$, assume $a+b<0$ and $-b\le c<0$. If $\mathcal{S}$ were a rotational surface whose generating curve cuts orthogonally the rotation axis, there would exist a point $p\in\mathcal{S}$ such that $\nu^2(p)=1$. Our assumption on $a,b$ and $c$ implies $a<-b\le c$, that is $c-a>0$. Thus in such a point $p\in\mathcal{S}$, we would have $K_e(p)=\dfrac{c-a}{a+b}<0$, a contradiction. This complete the proof. 
\end{proof}

\section{Vertical and Horizontal Height Estimates}\label{HVE}

In this section we consider a W-surface $\Sigma$ isometrically immersed in $\me$, having positive extrinsic curvature. Once the extrinsic curvature is positive, the surface is orientable and we orient $\Sigma$ in such way that the second fundamental form is positive definite. Let $z$ be a conformal local parameter for the second fundamental form, in this parameter the first and second fundamental form of $\Sigma$ are written as 
\begin{eqnarray}
\label{ae6}I&=& E dz^2+2 F \vert dz\vert^2+ \overline{E}d\overline{z}^2; \\
\label{ae7}II&=& 2 \rho \vert dz\vert^2,
\end{eqnarray}
where $\rho$ is a positive function and  $\overline{z}$ denotes the conjugate of $z$. The extrinsic curvature of $\Sigma$ is given by $K_e=\dfrac{\rho^2}{D}$, where $D=  F^2-\vert E\vert^2>0$, and we denote by $K_i$ the intrinsic curvature of the surface.

\subsection{Some basic equations} In this subsection we compute some equations which will be necessary to achieve the classification of W-surfaces.  

\begin{lemma}\label{al1}
Let $\varphi:\Sigma \fl \me$ be an isometric immersion in $\me$. Assume $\Sigma$ is a  W-surface having positive extrinsic curvature and that $a+b$ is different from zero. Let $N$ be the global unit normal vector field such that the second fundamental form of $\Sigma$ is positive definite and $z$ be a conformal parameter for the second fundamental form. Then, the following equations are satisfied: 
\begin{eqnarray}
\label{ae1} K_e&=&\dfrac{c-\epsilon a \nu^2}{a+b}\\[15pt]
\label{ae2}\dfrac{\rho_{\oz}}{\rho}&=&-\dfrac{\epsilon \nu\alpha}{\rho}-(\Gamma_{12}^1-\Gamma_{22}^2)\ \ \ \ \quad  \textnormal{(Codazzi equation)}\\[15pt]
\label{ae3} h_{z\oz}&=&\nu\rho+\Gamma_{12}^1h_z+\Gamma_{12}^2h_{\oz}\\[15pt]
\label{ae4} h_{zz}&=&\Gamma_{11}^1h_z+\Gamma_{11}^2h_{\oz}\\[15pt]
\label{ae13}\nu_{\oz}&=&-\dfrac{\alpha K_e}{\rho}\\[15pt]
\label{ae5} \vert T\vert^2&=& 1-\nu^2=\dfrac{1}{D}(\alpha h_z+\overline{\alpha} h_{\oz})\\
\label{ae27}\vert h_z\vert^2&=&-\dfrac{\vert\alpha\vert^2}{D}+ F\vert T \vert^2,
\end{eqnarray}
where, $\Gamma_{ij}^k, \ i.j,k=1,2,3$ are the Christoffel symbols associated to $z$; \  $E,F, \rho$ are terms of the fist and second fundamental forms given by equations \eqref{ae6} and \eqref{ae7}, and 
\begin{eqnarray}
\label{ae10}\alpha&:=& F  h_{\oz}-\overline{E} h_z\\
\label{ae11} D&:=& F^2-\vert E\vert^2\\
\label{ae00}T&=&\dfrac{1}{D}(\alpha \, \partial_z+\overline{\alpha}\,\partial_{\oz}).
\end{eqnarray}
\end{lemma} 

\begin{proof} This lemma is similar to \cite[Lemma 3.1]{AEG},  for completeness we present a proof here.  The idea is to write the compatibility equations in terms of the conformal parameter $z$. The compatibility equations for immersions in $\me$ are described in \cite{Benoit}. 

Let $\pi_2:\me\fl\real$, $\pi_2(p,t)=t$ be the projection on the second factor. We write $\dfrac{\partial}{\partial_t}=T+\nu N$, where $T$ is a tangent vector field to $\Sigma$. Once $\dfrac{\partial}{\partial_t}$ is the gradient in $\me$ of the function $\pi_2$, the vector field $T$, tangent to $\Sigma$, is the gradient of the height function $h:=\pi_2\vert_{\Sigma}$.  Then,   
\begin{equation*}
T=\dfrac{1}{D}(\alpha \, \partial_z+\overline{\alpha}\,\partial_{\oz}).
\end{equation*}

Observe that $\vert T\vert^2=1-\nu^2$, and after a direct computation we obtain equation \eqref{ae5}. On the other hand, by equation \eqref{ae10}  $h_z=\dfrac{1}{D}(E \,\alpha+ F\,\overline{\alpha})$, then
\begin{eqnarray*}
\vert h_z\vert^2&=& \dfrac{1}{D^2}(\vert\alpha\vert^2(\vert E\vert^2+F^2)+ F(E\alpha^2+\overline{E}\,\overline{\alpha}^2))\\[15pt]
&=&-\dfrac{\vert \alpha\vert^2}{D}+ \dfrac{F}{D^2}(E \alpha^2+2F\vert \alpha\vert^2+\overline{E}\,\overline{\alpha}^2)\\[15pt]
&=&-\dfrac{\vert \alpha\vert^2}{D}+F\vert T\vert^2,
\end{eqnarray*}
which proves equations \eqref{ae27}. 

Using the Gauss equation $K_i=K_e+\epsilon \nu^2$, the Weingarten equation $a K_i+b K_e=c$  \ becomes
\begin{equation*}
K_e=\dfrac{c-\epsilon a \nu^2}{a+b}.
\end{equation*}

The Codazzi equation is

\begin{equation}
\label{ae12} \nabla_X\mathcal{A}Y-  \nabla_Y\mathcal{A}X-\mathcal{A}\left[ X,Y\right]=\epsilon \nu (\langle Y,T\rangle X-\langle X,T\rangle Y), 
\end{equation}
where $\mathcal{A}$ is the shape operator of $\Sigma$ and $X, Y$  are tangent vector fields to $\Sigma$.  For $X=\partial_{\oz}, Y=\partial_z$ the Codazzi equation is,
\begin{equation*}
\nabla_{\partial_{\oz}}\mathcal{A}\partial_{z}-  \nabla_{\partial_{z}}\mathcal{A}\partial_{\oz}=\epsilon \nu (h_z \partial_{\oz}-h_{\oz} \partial_{z}), 
\end{equation*}
the scalar product of this equation with $\partial_{\oz}$ gives
\begin{equation*}
\langle\nabla_{\partial_{\oz}}\mathcal{A}\partial_{z}, \partial_{\oz}\rangle-  \langle \nabla_{\partial_{z}}\mathcal{A}\partial_{\oz},\partial_{\oz}\rangle=\epsilon \nu (h_z\overline{E} -h_{\oz}F)=-\epsilon\nu \alpha 
\end{equation*}
\[
\dfrac{\rho_{\oz}}{\rho}+(\Gamma_{12}^1-\Gamma_{22}^2)=-\dfrac{\epsilon \nu\alpha}{\rho},
\]
which is the equation \eqref{ae2}.

Taking the scalar product of the  compatibility equation $\nabla_X T=\nu \mathcal{A}X$ with $\partial_{\oz}$, for $X=\partial_z$ we get,
\begin{eqnarray*}
\langle \nabla_{\partial_z} T, \partial_{\oz}\rangle&=&\nu\langle \mathcal{A}\partial_z,\partial_{\oz}\rangle\\[15pt]
h_{z\oz}-\langle T,\nabla_{\partial_z}\partial_{\oz}\rangle&=& \nu\rho,
\end{eqnarray*}
then, we obtain equation \eqref{ae3}
\[
h_{z\oz}=\nu\rho+\Gamma_{12}^1h_z+\Gamma_{12}^2h_{\oz}.
\]
Similarly, taking the scalar product of the  compatibility equation $\nabla_X T=\nu \mathcal{A}X$ with $\partial_{z}$, for $X=\partial_z$ we get equation \eqref{ae4}.

From the compatibity equation $\mathrm{d}\nu(X)=-\langle \mathcal{A} X, T\rangle$, for $X=\partial_{\oz}$, we have
\[ 
\nu_{\oz}=-\langle \mathcal{A}\partial_{\oz}, \dfrac{\alpha \partial_z+\overline{\alpha}\partial_{\oz}}{D}\ \rangle=-\dfrac{\alpha K_e}{\rho}.
\]
\end{proof}

The equations on Lemma \ref{al1} enable us to rewrite $h_{z\oz}$ and $\nu_{z \oz}$, in a more suitable form, it is done in the following Proposition. 

\begin{proposition}\label{ap1} Let $\varphi:\Sigma \fl \me$ be an isometric immersion in $\me$. Assume $\Sigma$ is a  W-surface having positive extrinsic curvature and that $a+b$ is different from zero.  Let $N$ be the global unit normal vector field to $\Sigma$ such that the second fundamental form is positive definite and $z$ be a conformal parameter for the second fundamental form, then
\begin{eqnarray}
\label{ae8} h_{z\oz}&=&\dfrac{\nu\rho}{2 K_e(a+b)}(2 K_e(a+b)-\epsilon(2a+b)(1-\nu^2))\\[15pt]
\label{ae9} \nu_{z\oz}&=&-\dfrac{\epsilon a \nu \vert\alpha\vert^2}{(a+b)D}- F\nu K_e,
\end{eqnarray}
where $\alpha$ and $D$ are defined in \eqref{ae10} and \eqref{ae11}, respectively.
\end{proposition}

\begin{proof}
We start proving equation \eqref{ae8}. Since $\Sigma$ is a W-surface, taking the derivative of equation \eqref{ae1} with respect to $\oz$  and using equation \eqref{ae13}, we obtain 
\begin{equation}
\label{ae14} \dfrac{(K_e)_{\oz}}{2K_e}=\dfrac{\epsilon\, a\, \alpha \, \nu}{(a+b)\rho}.
\end{equation}
On the other hand, $K_e=\dfrac{\rho^2}{D}$, then
\begin{equation}
\label{ae15} \dfrac{(K_e)_{\oz}}{2K_e}=\dfrac{\rho_{\oz}}{\rho}-\dfrac{D_{\oz}}{2D}.
\end{equation}
As a consequence of \eqref{ae14} and \eqref{ae15}, we have
\begin{equation}
\label{ae16}
\dfrac{\rho_{\oz}}{\rho}-\dfrac{D_{\oz}}{2D}=\dfrac{\epsilon\, a\, \alpha \, \nu}{(a+b)\rho}.
\end{equation}
A direct computation, see \cite[Lemma 8]{TM1}, 
\begin{equation}
\label{ae17}
\Gamma_{12}^1+\Gamma_{22}^2=\dfrac{D_{\oz}}{2 D}.
\end{equation}
The Codazzi equation \eqref{ae2} is equivalent to 
\begin{eqnarray}
\nonumber \hspace{1cm}\dfrac{\rho_{\oz}}{\rho}-(\Gamma_{22}^2+\Gamma_{12}^1)+2 \Gamma_{12}^1&=& -\dfrac{\epsilon\, \nu\,\alpha}{\rho}\\[15pt]
\nonumber  {\textnormal{by\ \ }\eqref{ae17},\atop{\Longrightarrow}}\hspace{2,5cm}  \dfrac{\rho_{\oz}}{\rho}-\dfrac{D_{\oz}}{2D}+2 \Gamma_{12}^1&=& -\dfrac{\epsilon\, \nu\,\alpha}{\rho} \\[15pt]
\label{ae18} {\textnormal{by\ \ }\eqref{ae16},\atop{\Longrightarrow}}\hspace{4,5cm}\ \Gamma_{12}^1&=&-\dfrac{\epsilon\, \alpha\, \nu\,(2a+b)}{2\rho\,(a+b)}.
\end{eqnarray}

Since $\Gamma_{12}^1=\overline{\Gamma_{12}^2}$, using equation \eqref{ae3}, we have
\begin{eqnarray}
\nonumber h_{z\oz}&=&-\dfrac{\epsilon\, \nu\,(2a+b)}{2\rho\,(a+b)}(\alpha h_{z}+\overline{\alpha}h_{\oz})+\nu\,\rho\\[15pt]
\nonumber { \textnormal{by\ \ }  \eqref{ae5},\atop{\Longrightarrow}}\hspace{1,5cm}h_{z\oz}&=& -\dfrac{\epsilon\, \nu\,(2a+b)}{2\rho\,(a+b)}(1-\nu^2) D+\nu\,\rho\\[15pt]
\nonumber&=& \nu\,\rho\left( 1-\dfrac{\epsilon\,(2a+b)(1-\nu^2)}{2\, K_e\, (a+b)}\right)\\[15pt]
\nonumber&=&\dfrac{\nu\rho}{2 K_e(a+b)}(2 K_e(a+b)-\epsilon(2a+b)(1-\nu^2)),
\end{eqnarray}
which prove equation \eqref{ae8}.

In order to prove equation \eqref{ae9}, observe that by equations \eqref{ae13}, \eqref{ae14} and \eqref{ae16}, we have
\begin{eqnarray}
\nonumber \nu_{z\oz}&=& -\alpha_{z} \dfrac{K_e}{\rho}  -\dfrac{2\epsilon\, a \, \nu\,\vert\alpha\vert^2 K_e}{\rho^2\,(a+b)}+
\dfrac{\alpha K_e}{\rho}\left(\dfrac{\epsilon\, a \,\overline{\alpha}\,\nu}{\rho\, (a+b)}+\dfrac{D_z}{2D}\right)\\[15pt]
\label{ae19} &=& -\alpha_{z} \dfrac{K_e}{\rho}  -\dfrac{\epsilon\, a \, \nu\,\vert\alpha\vert^2 }{D\,(a+b)}+
\dfrac{\alpha K_e}{\rho}\dfrac{D_z}{2D}.
\end{eqnarray}
We claim that  
\begin{equation}
\label{ae20}\alpha_z=\alpha\dfrac{D_{z}}{D}+ F\,\nu\,\rho.
\end{equation}
Let us assume this equation for a moment. Then, a direct computation using equations \eqref{ae19} and \eqref{ae20} gives the equation \eqref{ae9}, as desired. So, in order to finish the proof of the proposition, we need to prove that equation \eqref{ae20} holds.  Recall $\alpha=F h_{\oz}- \overline{E}h_z$. Then, using equations \eqref{ae3} and\eqref{ae4}, we obtain  
\begin{eqnarray}
\nonumber\alpha_z&=& \langle \nabla_{\partial_{z}}\partial_{z}, \partial_{\oz}\rangle h_{\oz}+ \langle \nabla_{\partial_{z}}\partial_{\oz}, \partial_{z}\rangle h_{\oz}- 2\langle \nabla_{\partial_{z}}\partial_{\oz}, \partial_{\oz}\rangle h_z+F h_{z\oz}- \overline{E} h_{zz}\\[15pt]
\nonumber&=&\Gamma_{11}^1(F\, h_{\oz}-\overline{E}\,h_z)+\Gamma_{12}^1(E\, h_{\oz}-F\,h_z)+2\Gamma_{12}^2(F\, h_{\oz}-\overline{E}\,h_z)+F\,\nu\,\rho,\\[15pt]
\nonumber&=&\Gamma_{11}^1\,\alpha-\Gamma_{12}^1\,\overline{\alpha} +2\Gamma_{12}^2\alpha+F\,\nu\,\rho,
\end{eqnarray}
a direct computation using equation \eqref{ae18} shows that $\Gamma_{12}^2\alpha- \Gamma_{12}^1\,\overline{\alpha}=0$. Moreover,  conjugating equation \eqref{ae17}, we obtain
\[
\alpha_z=\alpha(\Gamma_{11}^1+\Gamma_{12}^2)+F\,\nu\,\rho=\alpha\left(\dfrac{D_z}{D}\right) +F\,\nu\,\rho,
\]
as claimed.

\end{proof}

\subsection{A quadratic form on $\Sigma$}\label{quadratic} In this section, we will define a quadratic form $Qdz^2$ on $\Sigma$ having the property that $Q$ vanishes identically or its zeros are isolated with negative index. 

Let $\Sigma$ be a W-surface isometrically immersed in $\me$ having positive extrinsic curvature, assume that $a+b$ and $2a+b$ are different from zero.  For such a W-surface we introduce the quadratic forms
\begin{eqnarray}
\label{ae21}A&:=& I+ f(1-\nu^2) \, dh^2,\\[15pt]
\label{ae22}Q\, dz^2&:=& (E+ f(1-\nu^2)\, h_z^2) \,dz^2, 
\end{eqnarray}
 where $I$ is the first fundamental form of $\Sigma$ given in \eqref{ae6} and $f:\left[ 0,1\right] \fl \real$ is the real 
 analytic function given by
\begin{equation}
\label{ae23}f(t)=\dfrac{-\epsilon(2a+b)(c-\epsilon a)t-(c-\epsilon a)^2+(c-\epsilon a(1-t))^{\frac{2a+b}{a}}(c-\epsilon a)^{-\frac{b}{a}}}{\epsilon(a+b)(c-\epsilon a)\,t^2}.
\end{equation}

\begin{remark}  We point out that
\begin{enumerate}
\item The quadratic form  $Q\, dz^2$ is the $(2,0)$-part of $A$.  
\item The Taylor series near zero of $f$ is 
\[
f(t)=\displaystyle\sum_{n=0}^{n=\infty} a_n\, t^n,  
\]
where $a_n= \dfrac{\epsilon^{n+1}}{(a+b)(c-\epsilon\,a)^{(1+n)}\,(n+2)!}\, \displaystyle\prod_{j=0}^{n+1}(2a+b-j\,a)$. The convergence radius of this series is  $\dfrac{\vert c-\epsilon\, a\vert}{\vert a\vert}>0$. So, $f$ is real analytic on $[0,1]$.
\end{enumerate}
\end{remark}

 The extrinsic curvature of the pair $(II, A)$ is, see \cite{TM1}
\begin{eqnarray}\label{ae40}
K(II,A)&=&\dfrac{(F+f(1-\nu^2)\, \vert h_z\vert^2)^2- \vert E+f(1-\nu^2)\, h_z^2\vert^2}{\rho^2}\\[10pt]
\nonumber &=& \dfrac{F^2-\vert E\vert^2}{\rho^2}-\dfrac{f(1-\nu^2)\, ( \overline{E} h_z^2+-2F \vert h_z\vert^2+ Eh_{\oz}^2)}{\rho^2}\\[10pt] \nonumber
 { \textnormal{by\ \ }  \eqref{ae10},  \atop{\Longrightarrow}}\hspace{2,5cm}&=& \dfrac{1}{K_e}+\dfrac{f(1-\nu^2)\, D(\alpha h_z+\overline{\alpha}h_{\oz})}{\rho^2}\\[10pt]\nonumber
 { \textnormal{by\ \ }  \eqref{ae5},\atop{\Longrightarrow}}\hspace{2,5cm}&=&\dfrac{1}{K_e}(1+f(1-\nu^2)\, \vert T\vert^2),
\end{eqnarray}
in particular, once $\vert Q\vert^2=\vert E+f(1-\nu^2)\, h_z^2\vert^2$, using the first and fourth lines of \eqref{ae40}, we have
\begin{equation}
\label{ae030}\vert Q\vert^2= (F+f(1-\nu^2)\, \vert h_z\vert^2)^2 -D(1+f(1-\nu^2)\, \vert T\vert^2).
\end{equation}

The next result is the key lemma, which gives a estimate of $\vert Q_{\overline{z}}\vert$ in terms of  the function $\vert Q\vert$, more precisely

\begin{lemma}\label{al2}
Let $\varphi:\Sigma\fl \me$ be an isometric immersion in $\me$. We assume $\Sigma$ is a W-surface having positive extrinsic curvature. We also suppose that $ a+b$ and $2a+b$ are different from zero. Let $z$ be a conformal parameter for the second fundamental form. Then 
\begin{equation}
\label{ae24}
\vert Q_{\oz}\vert\leq \dfrac{2 \vert \nu\, \rho\, h_z^3\, f^\prime(1-\nu^2)\vert }{D} \vert Q\vert,
\end{equation}
where $Q$ and  $D$ are defined in \eqref{ae22} and \eqref{ae11}, respectively,  and $f^\prime(t)$ is the derivative of $f$ at $t$. 
\end{lemma}

\begin{proof}
The derivative of the function $Q$ with respect to $\oz$ is
\begin{equation}
\label{ae25}
Q_{\oz}=E_{\oz}+2f(1-\nu^2)\, h_z h_{z\oz}-2\nu \nu_{\oz} f^\prime(1-\nu^2)\, h_z^2.
\end{equation}
Let us determine the expression of $E_{\oz}$. Observe that the Christoffel symbols with respect to the conformal parameter $z$ satisfies,  $\Gamma_{12}^1 =\overline{\Gamma_{12}^2}$.  Using equations \eqref{ae18} and \eqref{ae10}, we have
\begin{eqnarray*}
E_{\oz}&=& \partial_{\oz}\langle \partial_z,\partial_z\rangle= 2(\Gamma_{12}^1 E+\Gamma_{12}^2 F)\\[15pt]
&=&  -\dfrac{\epsilon\, \nu(2a+b)}{\rho(a+b)}(\alpha E+\overline{\alpha} F)\\[15pt]
&=&-\dfrac{\epsilon\, \nu(2a+b)}{\rho(a+b)} D\, h_z,
\end{eqnarray*}
then, since $K_e=\dfrac{\rho^2}{D}$, we obtain
\begin{equation}
\label{ae26} E_{\oz}=-\dfrac{\epsilon\, \nu(2a+b)\rho \, h_z}{K_e(a+b)}.
\end{equation}
By equations \eqref{ae8}, \eqref{ae13}, \eqref{ae25} and \eqref{ae26}, we have
\begin{eqnarray}
\nonumber Q_{\oz}&=&\nu \rho h_z\left( -\dfrac{\epsilon(2a+b)}{K_e(a+b)}+f(1-\nu^2)\left(\dfrac{2 K_e(a+b)-\epsilon(2a+b)(1-\nu^2)}{K_e(a+b)}\right)+f^\prime(1-\nu^2)\dfrac{2\alpha K_e h_z}{\rho^2}\right),
\end{eqnarray}
a direct computation shows that for $2a+b\neq 0$,  
\[
 -\dfrac{\epsilon(2a+b)}{K_e(a+b)}+f(1-\nu^2)\left(\dfrac{2 K_e(a+b)-\epsilon(2a+b)(1-\nu^2)}{K_e(a+b)}\right)=- (1-\nu^2) f^\prime(1-\nu^2).
\]
Using the above equation,  we obtain
\begin{eqnarray}
\nonumber Q_{\oz}&=&\nu\, \rho\, h_z\, f^\prime(1-\nu^2) \left( -(1-\nu^2)+  \dfrac{2\alpha\ h_z}{D}\right)\\[15pt]
\nonumber {\textnormal{by\  \ }\eqref{ae5},\atop{\Longrightarrow}}\hspace{1,6cm} &=& \nu\, \rho\, h_z\, f^\prime(1-\nu^2) \left(  \dfrac{\alpha\ h_z-\overline{\alpha} \, h_{\oz}}{D}\right)\\[15pt]
\nonumber {\textnormal{by\  \ }\eqref{ae10},\atop{\Longrightarrow}}\hspace{1,4cm} &=& \dfrac{\nu\, \rho\, h_z\, f^\prime(1-\nu^2)}{D} \left( E h_{\oz}^2 -\overline{E} h_z^2\right)\\[15pt]
\nonumber {\textnormal{by\  \ }\eqref{ae22},\atop{\Longrightarrow}}\hspace{1,4cm} &=& \dfrac{\nu\, \rho\, h_z\, f^\prime(1-\nu^2)}{D} \left( (Q-f(1-\nu^2)\, h_z^2)\, h_{\oz}^2 -(\overline{Q}-f(1-\nu^2)\,  h_{\oz}^2) \, h_z^2\, \right)\\[15pt]
\nonumber  &=& \dfrac{\nu\, \rho\, h_z\, f^\prime(1-\nu^2)}{D} \left( Q h_{\oz}^2 -\overline{Q} h_z^2\right),
\end{eqnarray}
then 
\[
\vert Q_{\oz}\vert\leq \dfrac{2 \vert \nu\, \rho\, h_z^3\, f^\prime(1-\nu^2)\vert }{D} \vert Q\vert,
\]
as desired. 

\end{proof}

Lemma \ref{al2} is used to apply \cite[Lemma 2.7.1]{Jost} and obtain an important property of the  function $Q$. 

\begin{proposition}\label{ap2}
Let $\varphi:\Sigma\fl \me$ be an isometric immersion in $\me$. Assume $\Sigma$ is a  W-surface having positive extrinsic curvature. Moreover, we suppose  $ a+b$ and $2a+b$ are different from zero.  Consider $\Sigma$ as a Riemann surface with the conformal structure induced by its second fundamental form. Then, the quadratic form $Qdz^2$, where $Q:\Sigma\fl\mathbb{C}$ is defined in \eqref{ae22}, vanishes identically or its zeros are isolated with negative index.      
\end{proposition}  

A direct consequence of Proposition \ref{ap2} is 

\begin{proposition}\label{ap3}
Let $\varphi:\Sigma\fl \me$ be an isometric immersion in $\me$. Assume $\Sigma$ is a  W-surface having positive extrinsic curvature. Moreover, we suppose  $ a+b$ and $2a+b$ are different from zero.  Consider $\Sigma$ as a Riemann surface with the conformal structure induced by its second fundamental form. If $\Sigma$ is a topological sphere, then the function $Q$ is identically null on $\Sigma$.   
\end{proposition}

\subsection{Vertical height estimates} This section is devoted to give a vertical height estimates for some W-surfaces, more precisely

\begin{theorem}[Vertical height estimates] \label{avhe}
Let $\varphi:\Sigma\fl \mathbb{H}^2\times\real$ be a compact graph on a domain $\Omega\subset \mathbb{H}^2$ whose boundary is contained in the slice $\h^2\times\{0\}$.  Assume $\Sigma$ is a  W-surface having positive extrinsic curvature. Moreover, suppose $2a+b$ is different from zero, $a+b>0$ and $c>0$. Then there exists a constant $C_0$ which depends only on $K_e, a, b, c$, such that the height function $h$, satisfies  $h(p)\leq C_0$ for all $p$ in $\Sigma$.  
\end{theorem}

\begin{proof}
The idea of this proof is to cook up a sub-harmonic function $\phi=h+g(\nu)$ on $\Sigma$ having non-positive boundary values, where $g:[-1,0]\fl\real$ is to be determined. In order to compute $\phi_{z\oz}$,  we calculated $(g(\nu))_{z\oz}$. Taking into account equations  \eqref{ae1}, \eqref{ae9}, \eqref{ae13} and \eqref{ae27}, we obtain
\begin{eqnarray}
\nonumber & (g(\nu))_{z\oz}&=(\nu_z g^\prime(\nu))_{\oz}=\nu_{z\oz}\, g^\prime(\nu)+\vert\nu_z\vert^2\, g^{\prime\prime}(\nu)\\[10pt]
\nonumber&=& \dfrac{\vert\alpha\vert^2}{D}\left(\dfrac{a\nu}{(a+b)}g^\prime(\nu)+K_e\,g^{\prime\prime}(\nu)\right) - F\, \nu\, K_e\, g^\prime(\nu)\\[15pt]
\nonumber&=& F\left( \left(\dfrac{ a\nu\,(1-\nu^2)}{(a+b)}-\nu K_e\right)\,g^\prime(\nu) + K_e(1-\nu^2)g^{\prime\prime}(\nu)\,\right)+ \vert h_z\vert^2 \left( -\dfrac{a\nu}{(a+b)}g^\prime(\nu)- K_e\,g^{\prime\prime}(\nu)\right)\\[15pt]
\nonumber&=& K_e\left( F\left( \left(\dfrac{  a \nu(1-\nu^2)}{(c+ a \nu^2)}-\nu\right)g^\prime(\nu) +(1-\nu^2)g^{\prime\prime}(\nu)\right)+ \vert h_z\vert^2 \left(  -\dfrac{ a\nu}{(c+ a \nu^2)}g^\prime(\nu)- g^{\prime\prime}(\nu)\right)\right).
\end{eqnarray}

Let $g:[-1,0]\fl\real$ be a real function whose derivative is  given  by 
\[g^{\prime}(t)= M\sqrt{\dfrac{(\frac{c+a}{c+a t^2})^{\frac{a+b}{a}}-1}{ (1-t^2)(c+a\,t^2)}},\]
 where $M$ is a constant depending only on $a,b,c$  defined by
\begin{equation}
M=
\left\{\begin{array}{cll}
\dfrac{\displaystyle\max_{\nu\in[-1,0]} \left(1+\dfrac{(2a +b)(1-\nu^2)}{2(c+ a\ \nu^2)}\right)}{\displaystyle\min_{\nu\in[-1,0]}\left(\sqrt{\dfrac{1}{c+ a \, \nu^2}\left(\dfrac{c+ a}{c+a\,\nu^2}\right)^{\frac{a+b}{a}}}\right)},&\textnormal{if} & \displaystyle\max_{\nu\in[-1,0]} \left(1+\dfrac{(2a +b)(1-\nu^2)}{2(c+ a\ \nu^2)}\right)>0;\\[15pt]
1&\textnormal{if} & \displaystyle\max_{\nu\in[-1,0]} \left(1+\dfrac{(2a +b)(1-\nu^2)}{2(c+a\ \nu^2)}\right)\leq0,
\end{array}\right.
\end{equation}
here $\displaystyle\max_{s\in[s_0,s_1]}(u(s))$ and $\displaystyle\min_{s\in[s_0,s_1]}(u(s))$ are the maximum and minimum  of the function $u(s)$ for $s$ in $[s_0,s_1]$.

A direct computation shows that 
\begin{equation}
\label{ae28}
 \left(\dfrac{ a \nu(1-\nu^2)}{(c+ a \nu^2)}-\nu\right)g^\prime(\nu) +(1-\nu^2)g^{\prime\prime}(\nu)=-\dfrac{\nu g^\prime(\nu)}{1+(1-\nu^2)\,  f(1-\nu^2)},
\end{equation}
where the real function $f(t)$ is defined in \eqref{ae23}. Also, we have 
\begin{equation}
\label{ae29}
-\dfrac{  a\nu}{(c+ a \nu^2)}g^\prime(\nu)- g^{\prime\prime}(\nu)= f(1-\nu^2) \left( -\dfrac{\nu g^\prime(\nu)}{1+(1-\nu^2)\,  f(1-\nu^2)}\right).
\end{equation}

Then, by \eqref{ae28} and \eqref{ae29}, 
\begin{equation}\label{ae30}
g_{z\oz}= - \dfrac{\nu\, K_e\, \, g^\prime(\nu)}{1+(1-\nu^2)\, f(1-\nu^2)}(F+\vert h_z\vert^2 \, f(1-\nu^2)).
\end{equation}

Observe that $(F+\vert h_z\vert^2 \, f(1-\nu^2))$ is positive on $\Sigma$, in fact,  since the extrinsic curvature $K(II,A)$ on equation \eqref{ae40} of the pair $(II,A)$ is positive, the quadratic form $A$ is positive definite or negative definite which implies that either $(F+\vert h_z\vert^2 \, f(1-\nu^2))$ is positive in $\Sigma$ or it is negative everywhere. At the highest point $h_z=0$ and we have  $(F+\vert h_z\vert^2 \, f(1-\nu^2))= F$ is positive, so we conclude that $(F+\vert h_z\vert^2 \, f(1-\nu^2))$ is positive in $\Sigma$.  

In order to compute $\phi_{z\oz}$ it is worth to write $g^\prime(t)$ as 
\[
g^\prime(t)=M\sqrt{\dfrac{1 +  (1 - t^2)\ f(1-t^2)}{ Ke\,((c+a\,t^2)-(a+b)(1-t^2)(1+(1-t^2)\,\,f(1-t^2)))}}\, ,
\]
 keeping this in mind, using  equations \eqref{ae8}, \eqref{ae30} and \eqref{ae030}, we obtain
\begin{eqnarray}
\nonumber\left(\phi\right)_{z\oz}
&=&\nu\left( \rho+\dfrac{ \rho(2a +b)(1-\nu^2)}{2 K_e(a+b)} -\dfrac{K_e\,  g^\prime(\nu)}{1+(1-\nu^2)\, f(1-\nu^2)}(F+\vert h_z\vert^2 \, f(1-\nu^2))\right)\\[15pt]
\nonumber&=&\nu\left( \rho+\dfrac{ \rho(2a +b)(1-\nu^2)}{2 K_e(a+b)} -\dfrac{ K_e\, g^\prime(\nu)}{1+(1-\nu^2)\, f(1-\nu^2)}  \sqrt{\vert Q\vert^2+\dfrac{\rho^2(1+(1-\nu^2)\, \, f(1-\nu^2))}{K_e}} \right)\\[15pt]
\nonumber&\geq & \nu\,\rho\left( 1+\dfrac{(2a +b)(1-\nu^2)}{2 K_e(a+b)} -\dfrac{ K_e\, g^\prime(\nu)}{1+(1-\nu^2)\, f(1-\nu^2)}  \sqrt{\dfrac{1+(1-\nu^2)\, \, f(1-\nu^2)}{K_e}} \right)\\[15pt]
\nonumber &=& \nu\,\rho\left( 1+\dfrac{(2a +b)(1-\nu^2)}{2 K_e(a+b)} - M\,\sqrt{\dfrac{1}{(c+a\,\nu^2)- (a+b)(1-\nu^2)(1+(1-\nu^2)\, \, f(1-\nu^2))}}\,\, \right)\\[15pt]
\nonumber &=&  \nu\,\rho\left( 1+\dfrac{(2a +b)(1-\nu^2)}{2(c-\epsilon\, a\ \nu^2)} - M\,\sqrt{\dfrac{1}{c+ a \, \nu^2}\left(\dfrac{c+ a}{c+ a\,\nu^2}\right)^{\frac{a+b}{a}}}\,\, \right),
\end{eqnarray}
so, the definition of $M$ implies that $\varphi_{z\oz}\geq 0$.  
 Taking 
 \[
g(\nu)=\displaystyle\int_{0}^\nu g^\prime(t)dt,  
 \]
we have $\Delta^{II}(h+g(\nu))=\dfrac{2}{\rho}(h+g(\nu))_{z\oz}\geq0$ in $\Sigma$, where $\Delta^{II}$ is the laplacian with respect to the second fundamental form. Moreover, $h+g(\nu)$ is non-positive on the boundary of $\Sigma$,  once $g^\prime(\nu)$ is non-negative and $\nu\le0$, then we have that $h+g(\nu)$ is non-positive everywhere. In particular, the maximum of the function $h $ is  
\[
C_0=\displaystyle\int_{-1}^0 g^\prime(t)dt.
\]
\end{proof}

\subsection{Horizontal heigh estimates}

In this section we will see that the horizontal height for  a class of compact embedded W-surfaces in $\h^2\times\R$ with boundary on a vertical plane is bounded.

Let $\varphi:\Sigma\ \to M^2(\epsilon)\times\R$ be an isometric immersion from a oriented surface $\Sigma$. Recall $\Sigma$ is a W-surface if the Weingarten function 
$$K_e-\dfrac{c-a\epsilon\nu^2}{a+b}$$
vanishes identically. Observe that we may regard $K_e-\dfrac{c-a\epsilon\nu^2}{a+b}=0$ as a second order partial differential equation. From this point of view, once $\nu$ depends only on the first derivative of the immersion, it can be showed that the partial differential equation $K_e-\dfrac{c-a\epsilon\nu^2}{a+b}=0$ is absolutely elliptic if $K_e>0$. So, if $\Sigma$ is a W-surface having positive extrinsic curvature the interior and the boundary maximum principle, in the sense of Hopf, hold.

Let $\varphi_j:\Sigma_j\ \to M^2(\epsilon)\times\R$, $j=1,2$, be two isometric immersions. Assume $\Sigma_j$ is a W-surface having positive extrinsic curvature. Let $N_j$ be the global unit normal vector field to $\Sigma_j$ such that the second fundamental form is positive definite. Let $p\in\Sigma_1\cap\Sigma_2$ and assume that $N_1(p)=N_2(p)$. Once $N_1(p)=N_2(p)$, there is a neighbourhood $U_j\subset\Sigma_j$ of $p$ such that $U_j$ is a graph in exponential coordinates of a function $f_j$ defined on a neighbourhood $\mathcal{D}$ of the origin of $T_p\Sigma_1=T_p\Sigma_2$ ($T_p\Sigma_j$ is the tangent plane of $\Sigma_j$ at $p$). Since the extinsic curvature of $\Sigma_j$ is positive, $f_j$ is a positive function (for $\mathcal{D}$ small enough). We say that $\Sigma_1$ is above $\Sigma_2$, which we denote by $\Sigma_1\ge \Sigma_2$, in a neighborhood of $p$ if $f_1\ge f_2$ in $\mathcal{D}$.

Under this notation, we have the following important theorem.   
\begin{theorem}[Hopf Maximum Principle, \cite{Hopf}] \label{t1} Let $\varphi_j:\Sigma_j\ \to M^2(\epsilon)\times\R$, $j=1,2$, be two isometric immersions. Assume $\Sigma_j$ is a W-surface having positive extrinsic curvature. Let $N_j$ be the global unit normal vector field to $\Sigma_j$ such that the second fundamental form associated to $\Sigma_j$ is positive definite.

Suppose that,
\begin{itemize}
\item[i)] $\Sigma_1$ and $\Sigma_2$ are tangent at an interior point $p\in\Sigma_1\cap\Sigma_2$, or
\item[ii)] there exists $p\in\partial\Sigma_1\cap\partial\Sigma_2$ such that both $T_p\Sigma_1=T_p\Sigma_2$ and $T_p\partial\Sigma_1=T_p\partial\Sigma_2$,
\end{itemize}
furthermore, suppose that the unit normal vector fields of $\Sigma_1$ and $\Sigma_2$ coincide at $p$. If $\Sigma_1\ge\Sigma_2$ in a neighbourhood $U_j\subset\Sigma_j$ of $p$, then $\Sigma_1=\Sigma_2$ in $U_1=U_2$.
\end{theorem}

In order to state the horizontal height estimate, recall a vertical plane in $\h^2\times\R$ is the product $\gamma\times\R$ of a complete geodesic $\gamma\subset\h^2$ with the real line $\R$.

\begin{theorem}[Horizontal height estimates]\label{t7}
Let $\varphi:\Sigma \to \h^2\times\R$ be an isometric immersion. Suppose $\Sigma$ is a compact embedded W-surface having positive extrinsic curvature whose boundary is contained in a vertical plane $P$. Moreover, assume $a+b>0$ and $c>0$. Then the distance from $\Sigma$ to $P$ is bounded; i.e., there exists a constant $c_0$ depending on $a,b,c$, independent of $\Sigma$, such that
$$dist(q,P)\le c_0, \hspace{.2cm} \textnormal{for all} \hspace{.2cm} q\in\Sigma.$$
\end{theorem}

Once the interior and boundary maximum principle hold for W-surfaces $\Sigma$ isometrically immersed in $\me$ having positive extrinsic curvature, the proof of \cite[Theorem 6.2]{EGH} applies to our setting with the exception that the proof use the maximum principle to compare $\Sigma$ to a surface $\Sigma_0$ that in our case is the rotational  topological sphere presented in Section \ref{HSphere}.

\section{Properly embedded W-surfaces with finite topology and one top end}\label{PES}

We begin this section with the following definition.

\begin{definition}\label{dd3}
Let  $\varphi:\Sigma \to \me$ be an isometric immersion from a complete surface. Then:
\begin{enumerate}
\item We say that $\Sigma$ has a top end $\mathcal{E}$ (respectively, a bottom end) if for any divergent sequence $\{q_j\}\subset\mathcal{E}$, the height function goes to $+\infty$ (respectively, $-\infty$).
\item For the case $\h^2\times\R$, we say that $\Sigma$ has a simple end if the boundary at infinity of the projection on the first factor $\pi_1(\Sigma)\subset\h^2\times\{0\}$ is a unique point $\theta_0$ and in addition, for each vertical plane $P$ whose boundary at infinity does not contain $\theta_0$, the intersection of $P$ and $\Sigma$ is either empty or a compact set. Here we are denoting by $\pi_1: \me \to  \m$, $\pi_1(p,t)=p$, the projection on the first factor; and as usual, we identify the base space $\m$ with its horizontal lift $\m\times\{0\}$.
\end{enumerate}
\end{definition}

Recall that, there is no properly embedded complete surface in $\h^2\times\R$ having positive constant extrinsic curvature with finite topology and one top (or bottom) end, see \cite[Theorem 7.2]{EGH}. In this section we extend this result to some W-surfaces.

For fixed real numbers $a$, $b$ and $c$ such that $a+b>0$ and $c>0$, we denote by $\mathcal{S}_c(a,b)$ the rotational topological sphere in $\h^2\times\R$ whose intrinsic and positive extrinsic curvatures satisfy the equation
\begin{equation}\label{e35}
a K_i+b K_e=c,
\end{equation} 
such rotational topological sphere was given in Section \ref{HSphere}. We denote by $c_1=2\kappa_0$ the horizontal diameter of $\mathcal{S}_c(a,b)$, where $\kappa_0$ satisfies
$$(\cosh \kappa_0)^{\frac{a}{a+b}}=\sqrt{\dfrac{a+c}{c}}.$$


The following lemma extend the Plane Separation Lemma given in \cite[Lemma 2.4]{NR} to properly embedded W-surface having positive extrinsic curvature. Using the Maximum Principle (Theorem \ref{t1}), the proof of Lemma \ref{l3} is similar to the one of \cite[Lemma 2.4]{NR}, so we will not present a proof here.

\begin{lemma}[Plane Separation Lemma]\label{l3} Let $\varphi:\Sigma\ \to \h^2\times\R$ be an isometric  properly embedded W-surface having positive extrinsic curvature. Assume $\Sigma$ has finite topology and a top (or bottom) end. Moreover, suppose $a+b>0$ and $c>0$. Let $P_1^+$ and $P_2^+$ be two disjoint half-spaces determined by vertical planes $P_1$ and $P_2$, respectively. If the distance between $P_1$ and $P_2$ is larger than the horizontal diameter $c_1$ of the rotational topological sphere $\mathcal{S}_c(a,b)$. Then, either $\Sigma\cap P_1^+$ or $\Sigma\cap P_2^+$ consist entirely of compact components.

\end{lemma}

As a consequence from Plane Separation Lemma and horizontal height estimates, we have the following theorem.
\begin{theorem}\label{t4}
 Let $\varphi:\Sigma\ \to \h^2\times\R$ be an isometric  immersion.  Assume $\Sigma$ is a complete W-surface having positive extrinsic curvature and finite topology with a top (or a bottom) end. Moreover, suppose $a+b>0$ and $c>0$. Then $\Sigma$ is contained in a vertical cylinder $\alpha\times\R$ in $\h^2\times\R$, where $\alpha\subset\h^2\times\{0\}$ is a circle.
\end{theorem}
\begin{proof} First, observe that, since $\Sigma$ has positive extrinsic curvature then $\Sigma$ is properly embedded, see \cite[Theorem 3.1]{EGH}. 
We take the disk model for the hyperbolic plane $\h^2$. Up to an isometry of the ambient space, we can assume that the point $\mathcal{O}=(\mathbf{0},0)$ belongs to $\Sigma$, here $\mathbf{0}$ denotes the origin of $\h^2$. Let $\gamma:[0,+\infty)\to \h^2\times\R$ be any horizontal geodesic starting at $\mathcal{O}$ parameterized by  arc length. We denote by $P(s)$, $s\in [0,+\infty)$, the vertical plane passing through $\gamma(s)$ orthogonal to $\gamma$.

$\mathbf{Claim}.-$ There exists a constant $c_2$, independent of $\gamma$, such that, if $s_0>c_2$, then the half-space determined by $P(s_0)$ that does not contain the point $\mathcal{O}$ is disjoint from $\Sigma$.

$\mathbf{Proof\hspace{.1cm} of \hspace{.1cm}the \hspace{.1cm}Claim}.-$We choose $R>\max\{c_0,c_1\}$, where $c_0$ and $c_1$ are the constant given by Theorem \ref{t7} and the Plane Separation Lemma, respectively. Denote by $P^+(R)$ the half-space determined by $P(R)$ containing the point $\mathcal{O}$ and by $P^-(2R)$ the half-space determined by $P(2R)$ which does not contain the point $\mathcal{O}$. By the Plane Separation Lemma applied to $\Sigma$, we have
\begin{itemize}
\item[a)] $\Sigma\cap P^+(R)$ has only compact components, or
\item[b)] $\Sigma\cap P^-(2R)$ has only compact components.
\end{itemize} 
By Theorem \ref{t7}, if $a)$ were true, the distance between the plane $P(R)$ and the point $\mathcal{O}\in \Sigma\cap P^+(R)$ would be at most $c_0$. Once this horizontal distance is $R>c_0$, $a)$ cannot occur. So $b)$ holds. Again, by Theorem \ref{t7}, the maximum distance between $\Sigma\cap P^-(2R)$ and the plane $P(2R)$ is at most $c_0$; hence $\Sigma$ is disjoint from the half-space determined by $P(2R+c_0)$ which does not contain the point $\mathcal{O}$. Choosing the constant $c_2=2\max\{c_0,c_1\}+c_0$, the Claim is proved.

The claim guarantees that $\Sigma$ is contained in the vertical cylinder $\alpha\times\R$, here $\alpha$ is a circle centered at the origin of $\h^2\times\{0\}$ having radius $c_2$.
\end{proof}
We finalize this section with a non-existence theorem.
\begin{theorem}\label{t5}
There is no complete W-surface having positive extrinsic curvature, with $a+b>0$, $2a+b\neq0$ and $c>0$, having finite topology and one top (or bottom) end in the product space $\h^2\times\R$. 
\end{theorem}
\begin{proof}
Assume by contradiction that there exists such a W-surface $\Sigma$ satisfying the hypothesis. Let $E(t)=\h^2\times\{t\}$ denote the horizontal slice at height $t$. By Theorem \ref{t4}, $\Sigma$ is contained in a vertical cylinder and since it has only one top (or bottom) end, $\Sigma$ is bounded from either above or below. Up to an isometry of the ambient space, we can assume that $\Sigma$ is bounded from below and a lowest point of $\Sigma$ lie in the slice $E(0)$. As reflections with respect to the slices $E(t)$ are isometries of $\h^2\times\R$, we perform Alexandrov reflection to $\Sigma$ with respect to the slice $E(t)$ for $t>0$. Since $\Sigma$ is contained in a vertical cylinder, the maximum principle ensures that no accident can occur moving $E(t)$ up, that is, for all $t\ge 0$, there is no point $p\in\Sigma\cap E(t)$ such that $\Sigma$ is orthogonal to $E(t)$. Moreover, denoting by $\Sigma^*_t$ the reflection of $\Sigma\cap(\h^2\times[0,t])$ with respect $E(t)$, by the maximum principle, for all $t>0$, there is no contact point between $\Sigma^*_t$ and $\Sigma\cap(\h^2\times(t,+\infty))$. Hence, for any $t>0$,  the part of $\Sigma$ below $E(t)$ is a vertical graph. But one can choose $t$ larger enough, so that we have a contradiction with Theorem \ref{avhe}.
\end{proof}

\section{The Main Theorem.}\label{SMT}
This last section is devoted to prove the main theorem.
\begin{theorem}\label{t6}
 Let $\varphi:\Sigma\ \to \me$ be an isometric  immersion.  Assume $\Sigma$ is a complete W-surface having positive extrinsic curvature. We suppose that $2a+b$ is different from zero, then $\Sigma$ is a topological rotational sphere described on Section \ref{HSphere} if 
\begin{enumerate}
\item[\emph{(A1)}] either $a+b>0$ and $c>0$,
\item[\emph{(A2)}] or  for $a+b<0$, 
\begin{enumerate}
\item[\emph{(i)}] $\epsilon=1$ and $c<-b$;
\item[\emph{(ii)}] $\epsilon=-1$ and $c<0$.
\end{enumerate}
\end{enumerate}
\end{theorem}

\begin{proof}
From \cite[Theorem 3.1]{EGH} and \cite[Theorem 2.4]{RE}, once $\Sigma$ has positive extrinsic curvature, $\Sigma$ is either homeomorphic to $\real^2$ or homeomorphic to $\s^2$. So, using  Lemma \ref{c12}, Theorem \ref{t7} and Theorem \ref{t5} if (A1) or (A2) is satisfied then $\Sigma$ is a topological sphere.  As a consequence,  Proposition \ref{ap3} says that the quadratic differential form $Q$ vanishes identically over $\Sigma$. 

Let $(u,v)$ be a local doubly orthogonal coordinates for the first and second fundamental form, in these coordinates
\[ I= {\bf E} \, du^2+ {\bf G}\, dv^2 \]
 \[II= \kappa_1{\bf E}\, du^2+ \kappa_2{ \bf G}\,  dv^2,\]
where $\kappa_1, \kappa_2$ are the principal curvatures of $\Sigma$. These coordinates are available on the interior of the set of umbilical points and also on a neighborhood of non umbilical points. So,  the set of  points where the coordinates $(u,v)$ are available is dense on $\Sigma$, thus,  properties obtained on this set are extended to $\Sigma$ by continuity.  

Since $Q$ vanishes on $\Sigma$, the quadratic form $A$, defined on \eqref{ae21} is conformal to  the second fundamental form $II$. It implies that   $h_u h_v=0$. Without lost of generality, we may assume that $h_u=0$ in the neighborhood where $(u,v)$ is available.  Then, since $A$ and $II$ are conformal and $h_u=0$,     
\begin{equation}
\label{ae50} 
A={\bf E} \, du^2+( {\bf G}+ h_v^2) \, dv^2=\dfrac{1}{\kappa_1}\, II.
\end{equation}

First, we prove that $\Sigma$ is invariant under a one parameter group of isometries. Second, we show that an orbit of this one parameter group of isometry  is a circle.   

Proceeding as in Lemma \ref{al1}, we write some compatibility equations with respect to the coordinates $(u,v)$, and we obtain
\begin{eqnarray}
\label{ae51} \nu_u&=&0\\[5pt]
\label{ae53}\dfrac{{\bf E}_v}{2{\bf E}}(\kappa_2-\kappa_1)&=&\epsilon \nu h_v +(\kappa_1)_v\\[5pt]
 \label{ae54}\dfrac{{\bf G}_u}{2{\bf G}}(\kappa_2-\kappa_1)&=&-(\kappa_2)_u\\[5pt]
 \label{ae55}h_{uv}&=& 0 = \dfrac{{\bf G}_u}{2{\bf G}} h_v.
\end{eqnarray}

From  equation \eqref{ae51} we obtain that $\nu$ does not depend on $u$. As the extrinsic curvature of $\Sigma$ is positive, no open neighbourhood of $\Sigma$ is contained in a slice, therefore, $h_v$ does not vanish in any open set where $(u,v)$ are available, then equation \eqref{ae55} implies $G_u=0$. Thus, by the Codazzi equation \eqref{ae54},   $(\kappa_2)_u=0$. Since the extrinsic curvature is $K_e=\dfrac{c-\epsilon a \nu^2}{a+b}$ and $\nu$ does not depend on $u$ neither does  $K_e$.  On the other hand, $K_e=\kappa_1 \, \kappa_2$ which implies that $(\kappa_1)_u=0$. 

The variables $(u,v)$ are available in the interior set of umbilical points and on a neighborhood of non umbilical points. Let us assume for a moment that we are working on a neighbohoof free of umbilical points. Then,  by the Codazzi equation \ref{ae53}, we may write  ${\bf E}={\bf E}_1(u){\bf E}_2(v)$. Considering the new variables
\[
x:= \sqrt{{\bf E}_1(u)}\,du  \hspace{2cm} y:=v,
\]
we conclude that the first and second fundamental forms of $\Sigma$, $h$ and $\nu$ depend only on $y$. Then,  $\varphi(x,y)$ and $\varphi(x+x_0,y)$ only differ by an isometry of the ambient space, in other words,  the immersion is invariant under one parameter group of isometries of the ambient space, given by the transformation $(x,y)\mapsto (x+t,y)$, see \cite{Benoit}. Once we know that $\Sigma$ is a topological sphere, we conclude that $\Sigma$ is invariant by the group of rotations of $M^2(\epsilon)$.

It remains to analyse the case where the coordinates $(u,v)$ are defined on a neighbourhood in the interior of umbilical points. In this case, 
\begin{eqnarray*}
I&=& {\bf E} du^2+{\bf G} dv^2\\
II&=& \kappa_1( {\bf E} du^2+{\bf G} dv^2)\\
A&=&  {\bf E} du^2+ ({\bf G} + f(1-\nu^2)\, h_v^2)dv^2= \dfrac{1}{\kappa_1} II.
\end{eqnarray*}
 In particular,  ${\bf G} + f(1-\nu^2) h_v^2= {\bf G}$ which implies that $h_v$ vanishes identically in this neighborhood. Then, once we are working on a neighborhood of umbilical points we conclude that the height function is constant, which implies that $\Sigma$ is contained in a slice. This gives a contradiction, since the extrinsic curvature of the surface is positive. Then, there is no such a neighborhood of umbilical points. 
\end{proof}

\vspace{1cm}

\begin{tabular}{l|l}
Abigail Folha  &  Carlos Pe\~{n}afiel\\
abigailfolha@vm.uff.br & penafiel@im.ufrj.br\\
Universidade Federal Fluminense\ &  Universidade Federal de Rio de Janeiro\\
Instituto de Matem\'{a}tica e Estat\'{i}stica & Instituto de Matem\'{a}tica e Estat\'{i}stica\\

Departamento de Geometria & Departamento de M\'{e}todos Matem\'{a}ticos \\

R. M\'{a}rio Santos Braga, s/n   &  Av. Athos da Silveira Ramos 149\\
Campus do Valonguinho & CT, bl. C, Cidade Universit\'{a}ria\\
CEP 24020-140  &  CEP 21941-909\\
Niter\'{o}i, RJ - Brasil. & Rio de Janeiro, RJ - Brasil.\\

\end{tabular}

\end{document}